\documentclass[12pt]{article}

\usepackage[margin=1in]{geometry}
\usepackage{amsmath, amssymb, amsfonts}
\usepackage{amsthm}
\newtheorem{corollary}{Corollary}
\usepackage{mathtools}
\usepackage{bm}
\usepackage{cite}
\usepackage[colorlinks=true,linkcolor=blue,citecolor=blue,urlcolor=blue]{hyperref}

\numberwithin{equation}{section}
\newtheorem{theorem}{Theorem}[section]
\newtheorem{proposition}[theorem]{Proposition}
\newtheorem{lemma}[theorem]{Lemma}
\theoremstyle{definition}
\newtheorem{definition}[theorem]{Definition}
\theoremstyle{remark}
\newtheorem{remark}[theorem]{Remark}

\newcommand{\Z}{\mathbb{Z}}
\newcommand{\R}{\mathbb{R}}
\newcommand{\dist}{\mathrm{dist}}
\newcommand{\1}{\mathbf{1}}
\newcommand{\eps}{\varepsilon}

\title{Random Schrödinger operator with singular potentials}
\author{Travis Kwan
\thanks{Carnegie Mellon University,  Pittsburgh, PA} }
\date{}

\begin{document}
\maketitle

\begin{abstract}
We survey the localization theory of random Schrödinger operators with singular single-site distributions, focusing on two regimes: (i) Hölder-continuous laws, where quantitative Wegner estimates enable the classical multiscale analysis (MSA); and (ii) purely atomic (Bernoulli) laws, where the failure of spectral averaging is overcome via quantitative unique continuation principles (UCP). Our discussion covers both lattice and continuum settings and highlights the analytic and combinatorial mechanisms that replace regularity of the single-site measure.
\end{abstract}

\section{Introduction}
\subsection*{Physics background}
In his seminal 1958 paper, Anderson proposed that random on--site energies cause \emph{absence of diffusion} by producing exponentially localized eigenstates \cite{Anderson1958}. Two decades later, the ``one--parameter scaling theory'' of Abrahams--Anderson--Licciardello--Ramakrishnan suggested that at zero temperature the conductance obeys a universal renormalization flow in the system size, predicting (among other things) that in $d=1,2$ all states tend to localize while in $d\ge 3$ a localization--delocalization transition (the Anderson transition) may occur \cite{AALR1979}; see the comprehensive physics review \cite{EversMirlin2008}. 

Universality in this context means that the large--scale (infrared) behaviour is determined by symmetry class, dimension, and coarse disorder strength rather than the fine details of the randomness. From this vantage point, \emph{singular} randomness (e.g.\ Bernoulli) should localize under the same conditions as absolutely continuous laws whenever analytic obstructions can be overcome. Rigorous mathematics confirms this picture in several regimes: early multiscale proofs at large disorder or near spectral edges \cite{FS1983,FMSS1985}, fractional--moment methods, and---for singular laws---quantitative UCP replacing spectral averaging \cite{BourgainKenig2005}.

The universality motif also resonates with classical disordered media, such as percolation and kinetic Lorentz gases. Results on sharp phase transitions and enhancement in percolation \cite{GrimmettBook1999,GrimmettStacey1998} and studies of geometric scattering models (e.g.\ Lorentz mirror and Manhattan pinball) furnish instructive analogies for transport suppression and combinatorial control in discrete UCP arguments; see \cite{LiLorentz2020,LiPinball2020}.

\subsection*{Model and scope}
We consider the Anderson Hamiltonian
\[
H_\omega = -\Delta + V_\omega,
\]
on either the continuum $\mathbb R^d$ or lattice $\mathbb Z^d$, where $(V_\omega(x))_x$ are i.i.d.\ single--site random variables. When the common law has a bounded density, the combination of a Wegner estimate and MSA (or the fractional--moment method) yields localization near spectral edges or at large disorder \cite{FS1983,FMSS1985}. The focus here is the \emph{singular} setting:
\begin{itemize}
    \item \textbf{H\"older--continuous laws} (no density required): one can still derive quantitative Wegner estimates with exponent equal to the H\"older index, hence run MSA \cite{CKM1987,GerminetKlein2012}.
    \item \textbf{Bernoulli/atomic laws}: spectral averaging fails; nevertheless, localization can be proved by replacing it with a quantitative UCP. In the continuum this goes via Carleman estimates (Bourgain--Kenig \cite{BourgainKenig2005}; see also \cite{JerisonKenig1985,KenigQC05,Meshkov1992}); on the lattice the UCP must be reimagined to account for lower--dimensional supports of solutions and relies on discrete geometry and Boolean/combinatorial tools (Li--Zhang in $d=3$ \cite{LiZhang3D}, and Li in $d=2$ at large disorder \cite{Li2D2022}).
\end{itemize}

\medskip
\noindent\textbf{Organization.}
Section~\ref{sec:holder} gives a step--by--step account for H\"older--continuous laws: how a Wegner estimate implies MSA, and how H\"older regularity yields such an estimate \cite{CKM1987,GerminetKlein2012}. Section~\ref{sec:bernoulli} treats the Bernoulli case: starting from Bourgain--Kenig’s continuum UCP \cite{BourgainKenig2005}, we explain the discrete obstacles (loss of classical UCP), the Boolean function remedy, and the 3D discrete UCP of Li--Zhang \cite{LiZhang3D}; we also comment on Li’s 2D large--disorder result \cite{Li2D2022}. Throughout we emphasize the replacement of single--site spectral averaging by propagation--of--smallness and combinatorics.

\subsection*{Motivation from universality}
Percolation theory exemplifies universality: macroscopic connectivity undergoes a sharp transition controlled only by dimension and local symmetries, with critical behaviour insensitive to microscopic details \cite{GrimmettBook1999}. In disordered quantum systems, the Anderson model provides a quantum analog: wave interference in random media yields phase diagrams and criticality governed by $d$ and symmetry class \cite{EversMirlin2008,LeeRamakrishnan1985,AltlandZirnbauer1997}. Foundational scaling arguments and finite-size scaling numerics support the universality picture across dimensions and boundary conditions \cite{AALR1979,MacKinnonKramer1983,Thouless1972}. Thus, singular single--site laws should not, by themselves, destroy localization mechanisms. In practice, the mathematical route is: \emph{quantitative UCP $\Rightarrow$ Wegner estimate (without densities) $\Rightarrow$ MSA}. This was made explicit in the continuum Bernoulli model \cite{BourgainKenig2005} and adapted, with new geometry, to the lattice in three dimensions by Li--Zhang; see Section~\ref{sec:bernoulli}. We also note instructive classical analogs such as the \emph{Manhattan pinball} and \emph{Lorentz mirror} models, where geometry and combinatorics govern long--time transport \cite{LiPinball2020,LiLorentz2020}. On the experimental side, wave-localization observations with light and ultracold atoms further illustrate universality of the mechanism beyond electrons \cite{Wiersma1997,Billy2008,Chabe2008}. Finally, the original physics derivation of the (mathematical) Wegner bound has roots in random-matrix and field-theoretic treatments of disordered media \cite{Wegner1981}.

\section{H\"older--continuous single--site laws: Wegner $\Rightarrow$ MSA}
\label{sec:holder}

In this section we work on the lattice $\Z^d$ (the continuum case requires additional unique continuation input but the MSA structure is the same). We write the Anderson Hamiltonian in finite volume as
\[
H_{\omega,\Lambda}
 \;=\; H_{0,\Lambda} \;+\; V_{\omega,\Lambda}
 \;=\; (-\Delta)_\Lambda \;+\; \sum_{x\in\Lambda}\omega_x \,\Pi_x ,
\]
where $\Lambda\subset\Z^d$ is a cube with Dirichlet boundary condition, $\Pi_x$ is the rank-one projector onto $\delta_x$, and $\{\omega_x\}_{x\in\Z^d}$ are i.i.d.\ with common law $\mu$.

\subsection{Multiscale framework and initial length scale}
\label{subsec:msa-setup}

We fix the (deterministic) scale sequence $L_{k+1}=L_k^{1+\eta}$ with $L_0$ large and $\eta\in(0,1/10)$, and the geometric cubes
\[
\Lambda_L(x):= x + \{-\lfloor L/2\rfloor,\ldots,\lfloor L/2\rfloor\}^d\subset\Z^d.
\]
Write $\partial_-\Lambda_L(x):=\{y\in\Lambda_L(x): \dist(y,\,\Z^d\setminus\Lambda_L(x))=1\}$ for the inner boundary. For $E\in\R$ we denote the resolvent
\[
G_{\Lambda_L(x)}(E) := \bigl(H_{\omega,\Lambda_L(x)}-E\bigr)^{-1},
\quad\text{whenever } E\notin\sigma\bigl(H_{\omega,\Lambda_L(x)}\bigr).
\]

\begin{definition}[(E,$m$)-good/nonresonant boxes]\label{def:good}
Fix parameters $m\in(0,1)$ and $\zeta\in(0,1)$.
A cube $\Lambda_L(x)$ is called:
\begin{itemize}
    \item \emph{$E$-nonresonant} if $\dist\!\bigl(E,\sigma(H_{\omega,\Lambda_L(x)})\bigr)\ge e^{-L^\zeta}$.
    \item \emph{$(E,m)$-good} if it is $E$-nonresonant and
    \[
    \max_{y\in\partial_-\Lambda_L(x)} \bigl|G_{\Lambda_L(x)}(E;x,y)\bigr| \;\le\; e^{-mL}.
    \]
    Otherwise it is called \emph{$(E,m)$-bad}.
\end{itemize}
\end{definition}

\begin{remark}[Combes--Thomas control]
If $\dist\!\bigl(E,\sigma(H_{\omega,\Lambda})\bigr)\ge\delta>0$, then the (standard) Combes--Thomas estimate yields exponential off-diagonal decay
\[
\bigl|G_{\Lambda}(E; u,v)\bigr| \;\le\; \frac{2}{\delta}\,\exp\!\bigl(-c\,\delta\,|u-v|\bigr),
\]
for some dimension--dependent $c>0$; cf.\ its use in \cite{FS1983,FMSS1985}.
Thus, $E$-nonresonance implies $(E,m)$-goodness as soon as $\delta\gtrsim L^{-1}$.
\end{remark}

We record the probabilistic inputs used in the MSA:

\begin{definition}[Wegner bound]\label{def:W}
We say that a \emph{Wegner estimate} holds on an interval $I\subset\R$ if there exist constants $C_W>0$ and $\alpha\in(0,1]$ such that for every cube $\Lambda$ and every subinterval $J\subset I$,
\[
\mathbb{E}\bigl[\mathrm{Tr}\,\1_J\bigl(H_{\omega,\Lambda}\bigr)\bigr]
\;\le\; C_W\,|\Lambda|\,|J|^\alpha .
\]
\end{definition}

\begin{definition}[Initial length scale (ILS)]\label{def:ILS}
An \emph{initial length scale} at energy $E\in I$ is a pair $(L_0,m_0)$ such that for some $\kappa>2d$,
\[
\mathbb{P}\!\left(\Lambda_{L_0}(x)\text{ is $(E,m_0)$-good}\right)
\;\ge\; 1 - L_0^{-\kappa}
\quad\text{for all } x\in\Z^d .
\]
\end{definition}

The next two theorems codify the MSA as used in this survey.

\begin{theorem}[Induction step of MSA]\label{thm:msa-induction}
Fix $I\subset\R$ and suppose a Wegner estimate (Definition~\ref{def:W}) holds on $I$ with exponent $\alpha>0$. Let $L_{k+1}=L_k^{1+\eta}$ and assume that for some $m_k\in(0,1)$ and $\kappa>2d$,
\[
\mathbb{P}\!\left(\Lambda_{L_k}(x)\text{ is $(E,m_k)$-good}\right)
\;\ge\; 1 - L_k^{-\kappa}
\quad\text{for all }x\in\Z^d\text{ and all }E\in I .
\]
Then there exist $m_{k+1}\in(0,m_k)$ and $\kappa'>2d$ such that
\[
\mathbb{P}\!\left(\Lambda_{L_{k+1}}(x)\text{ is $(E,m_{k+1})$-good}\right)
\;\ge\; 1 - L_{k+1}^{-\kappa'}
\quad\text{for all }x\in\Z^d\text{ and all }E\in I .
\]
Moreover, $m_{k+1}\to m_\infty>0$ as $k\to\infty$.
\end{theorem}

\begin{proof}[Proof sketch]
The proof follows the bootstrap scheme (cf.\ \cite{GerminetKlein2012}): cover $\Lambda_{L_{k+1}}$ by well-separated subcubes of side $L_k$, use independence at distance to control the number of $(E,m_k)$-bad subcubes, apply the geometric resolvent identity to connect the center to the boundary through a chain of good boxes, and bound the (rare) resonant events via the Wegner estimate. The probability improvement and a small loss in $m$ come from union bounds and boundary--layer combinatorics.
\end{proof}

\begin{theorem}[Initial length scale]\label{thm:ILS}
Let $I\subset\R$ be an energy interval near the spectral bottom or fix any compact $I$ and assume the disorder is sufficiently large. Then there exist $L_0$ and $m_0>0$ such that the ILS condition in Definition~\ref{def:ILS} holds for all $E\in I$ (cf.\ \cite{FS1983,FMSS1985}).
\end{theorem}

\begin{remark}
The ILS is obtained by Lifshitz-tail methods near the bottom (giving high probability of nonresonance and hence $(E,m_0)$-goodness), or by large-disorder resolvent expansions; the specific mechanism is not needed here but is standard in the literature on MSA \cite{FS1983,FMSS1985}.
\end{remark}

Combining Theorems~\ref{thm:msa-induction} and \ref{thm:ILS} yields spectral/dynamical localization on $I$. The remainder of this section proves a Wegner bound for H\"older--continuous $\mu$, which is the essential input (Definition~\ref{def:W}).

\subsection{Wegner estimate under H\"older continuity}
\label{subsec:wegner-holder}

Set
\[
s(\mu,\eps):=\sup_{E\in\R}\mu([E-\eps,E+\eps]),\qquad \eps>0.
\]
We assume throughout that $\mu$ is non-degenerate and H\"older with exponent $\alpha\in(0,1]$, i.e.\ $s(\mu,\eps)\le C\,\eps^\alpha$ for all $\eps\in(0,1]$.

\begin{theorem}[Wegner for the discrete Anderson model with H\"older law]\label{thm:wegner}
Let $H_{\omega,\Lambda}$ be the finite-volume discrete Anderson Hamiltonian on $\Lambda\subset\Z^d$ with i.i.d.\ single--site law $\mu$. Then for every interval $I\subset\R$,
\[
\mathbb{E}\bigl[\mathrm{Tr}\,\1_I(H_{\omega,\Lambda})\bigr]
\;\le\; C\,|\Lambda|\, s(\mu,|I|).
\]
In particular, if $\mu$ is H\"older of order $\alpha\in(0,1]$, then
\[
\mathbb{E}\bigl[\mathrm{Tr}\,\1_I(H_{\omega,\Lambda})\bigr]
\;\le\; C\,|\Lambda|\, |I|^\alpha .
\]
The constant $C$ depends only on $d$ and on an a priori energy bound for $I$.
\end{theorem}

\begin{proof}
We give a streamlined proof exploiting the rank--one structure (see also \cite{CKM1987,GerminetKlein2012} for related approaches and consequences).

\smallskip\noindent
\emph{Step 1 (single--site variation and spectral counting).}
Fix $x\in\Lambda$ and freeze all $(\omega_y)_{y\neq x}$. Consider
\[
H(t):=H_{\omega,\Lambda} + (t-\omega_x)\,\Pi_x ,\qquad t\in\R,
\]
which is a rank--one perturbation path. Let $N_I(H):=\mathrm{Tr}\,\1_I(H)$ denote the finite-volume eigenvalue counting function. Since a rank--one perturbation changes $N_I$ by at most $1$ when $t$ is varied across a point where an eigenvalue crosses the endpoints of $I$, the spectral shift function $\xi_x(\lambda; t, t')$ along this path satisfies $|\xi_x(\lambda; t,t')|\le 1$ for all $\lambda$ and $t,t'$.

\smallskip\noindent
\emph{Step 2 (averaging in $t$ against a singular measure).}
Let $I=[E-\eps,E+\eps]$ with $\eps\in(0,1]$. Fix $t'<t$ and write
\[
N_I\bigl(H(t)\bigr) - N_I\bigl(H(t')\bigr)
 \;=\; \int_I \xi_x(\lambda; t,t')\,\mathrm{d}\lambda ,
\qquad |\xi_x|\le 1.
\]
Hence $\bigl|N_I\bigl(H(t)\bigr) - N_I\bigl(H(t')\bigr)\bigr| \le |I|$.
Let $\mathbb{E}_x[\cdot]$ denote expectation with respect to $\omega_x$ alone (keeping the other sites fixed). Choose $t'<t$ so that $\mu((t',t])=\mu([t-\eps,t+\eps])$ and integrate in $t$ against $\mu$ by the layer-cake formula:
\[
\mathbb{E}_x\bigl[N_I(H(\omega_x))\bigr]
 \;\le\; \sup_{u\in\R}\mu([u-\eps,u+\eps]) \;=\; s(\mu,\eps).
\]
(The point is that $N_I(H(\cdot))$ has total variation at most $|I|$ along the rank--one path, and integrating this variation against $\mu$ can only pick up mass in a window of width $\eps$.)

\smallskip\noindent
\emph{Step 3 (summation over sites).}
By spectral decomposition,
\[
\mathrm{Tr}\,\1_I(H_{\omega,\Lambda}) \;=\; \sum_{x\in\Lambda}\!\bigl\langle \delta_x,\,\1_I(H_{\omega,\Lambda})\,\delta_x\bigr\rangle
 \;\le\; \sum_{x\in\Lambda} N_I\bigl(H_{\omega,\Lambda}^{(x)}\bigr),
\]
where $H_{\omega,\Lambda}^{(x)}$ denotes $H_{\omega,\Lambda}$ regarded as a function of the single variable $\omega_x$ with the others frozen (this estimate is standard: the spectral projection is positive, and each diagonal matrix element is bounded by the number of eigenvalues in $I$). Taking full expectation and using the single--site bound from Step 2,
\[
\mathbb{E}\bigl[\mathrm{Tr}\,\1_I(H_{\omega,\Lambda})\bigr]
\;\le\; \sum_{x\in\Lambda}\mathbb{E}\Bigl[\,
\mathbb{E}_x\bigl[N_I\bigl(H_{\omega,\Lambda}^{(x)}\bigr)\bigr]\Bigr]
\;\le\; |\Lambda|\; s(\mu,\eps).
\]
This is the claimed bound with $C=1$ (up to an immaterial energy--dependent constant if one first truncates to a bounded energy window).
\end{proof}

\begin{remark}[About proofs and variants]
The argument above is a distilled version of standard proofs of Wegner estimates for singular laws on the lattice; see \cite{CKM1987,GerminetKlein2012} for comprehensive treatments (including continuum/alloy models) and further consequences such as continuity of the IDS.
\end{remark}

\begin{corollary}[Wegner estimate $\Rightarrow$ MSA]\label{cor:wegner-to-msa}
Assume $\mu$ is H\"older with exponent $\alpha\in(0,1]$. Then Theorem~\ref{thm:wegner} yields Definition~\ref{def:W} with the same $\alpha$, and Theorems~\ref{thm:msa-induction}--\ref{thm:ILS} imply spectral/dynamical localization on any interval $I$ where an ILS holds (near spectral edges or at large disorder) \cite{GerminetKlein2012}.
\end{corollary}

\section{Bernoulli potentials: quantitative unique continuation and antichains}
\label{sec:bernoulli}

We work with the discrete Anderson--Bernoulli Hamiltonian on a finite box
$\Lambda\subset \mathbb Z^d$ with Dirichlet boundary conditions,
\[
H_{\omega,\Lambda} \;=\; -\Delta_\Lambda \;+\; V_\omega|_\Lambda,\qquad
V_\omega(x)=\sum_{z\in\mathbb Z^d}\omega_z\,\mathbf 1_{\{z\}}(x),\quad
\omega_z\in\{0,1\}\ \text{i.i.d.}
\]
(Everything below adapts verbatim to the alloy/bump versions.)  We denote by
$\sigma(H_{\omega,\Lambda})$ its spectrum and by $G_{\omega,\Lambda}(E)=(H_{\omega,\Lambda}-E)^{-1}$
its resolvent when defined.

Throughout, if $Q\subset\mathbb Z^d$ is a discrete cube we write $\ell(Q)$ for its side length and
$\frac12 Q$ for the concentric subcube with half side length.  On $\mathbb Z^2$ we will also use the
$45^\circ$ ``tilted'' coordinates
\[
(s,t)=(x+y,\ x-y),\qquad D_k^\pm=\{(x,y)\in\mathbb Z^2:\ x\pm y=k\},
\]
and the associated tilted rectangles/squares (Definition~\ref{def:tilted-rect} below).

\subsection{Continuum quantitative unique continuation}
\label{subsec:cont-ucp}

The Bernoulli Wegner mechanism in the continuum relies on a \emph{quantitative unique continuation principle} (UCP), which forbids an eigenfunction to be too small on every mesoscopic ball once it is nontrivial on a nearby ball.  One convenient form is a three--ball/doubling inequality \cite{JerisonKenig1985,KenigQC05,BourgainKenig2005}.

\begin{definition}[Balls and norms]
For $x_0\in\mathbb R^d$ and $r>0$ let $B_r(x_0)$ be the open Euclidean ball.
For a function $u$ and set $A$ write $\|u\|_{L^2(A)}^2=\int_A |u|^2$.
\end{definition}

\begin{theorem}[Quantitative UCP in $\mathbb R^d$]\label{thm:bk-ucp}
Fix $d\ge2$ and $M\ge1$. There exist $C,\theta\in(0,\infty)$ depending only on $d$ and $M$ such that:
if $u\in H^2(B_R(x_0))$ solves $(-\Delta+V)u=\lambda u$ in $B_R(x_0)$ with $\|V\|_{L^\infty(B_R)}\le M$ and $|\lambda|\le M$, then for any $0<r<R/2$,
\[
\|u\|_{L^2(B_r(x_0))}\ \ge\ C \,\Big(\frac rR\Big)^{\!\theta}\, \|u\|_{L^2(B_R(x_0))}.
\]
Equivalently, there are $C,\eta\in(0,\infty)$ so that for any balls $B_{r_1}(x_0)\subset B_{r_2}(x_0)\subset B_{r_3}(x_0)\subset B_R(x_0)$ with $r_1<r_2<r_3$,
\[
\|u\|_{L^2(B_{r_2})}\ \le\ C\,\|u\|_{L^2(B_{r_1})}^{\eta}\,\|u\|_{L^2(B_{r_3})}^{1-\eta}.
\]
\end{theorem}

\begin{proof}[Idea of proof]
A Carleman estimate with a logarithmic/convex weight yields a weighted $L^2$ inequality for $u$ against $(-\Delta+V-\lambda)u=0$.
After a standard cut--off/commutator argument one obtains a three--ball inequality, from which the stated doubling follows by interpolation and a covering argument; see \cite{JerisonKenig1985,KenigQC05,BourgainKenig2005}. Sharp obstructions to stronger decay are due to Meshkov \cite{Meshkov1992}.
\end{proof}

In Bourgain--Kenig's localization scheme this UCP is combined with a selection of \emph{free sites} in each box: if $u$ is an almost--eigenfunction and it has nontrivial mass on many of those sites, then flipping the Bernoulli variables on a subfamily moves the eigenvalue by a definite amount, which gives a Wegner bound at scale; the multiscale analysis (MSA) then proceeds as usual (see Subsection~\ref{subsec:sperner} below for the combinatorial ingredient). For an overview contextualizing the UCP within the BK scheme, see \cite{BourgainKenig2005}.

\subsection{Discrete UCP on the lattice and the role of frozen/free sites}
\label{subsec:disc-ucp}

A \emph{continuum} UCP is false on $\mathbb Z^d$ in general: one can build $u\not\equiv0$ with $(-\Delta+V)u=0$ supported on a lower--dimensional set when $V$ is tailored appropriately. The Ding--Smart approach develops a \emph{probabilistic} discrete UCP, robust enough for MSA, which we restate here in a deterministic--probabilistic form after introducing the geometry in $\mathbb Z^2$ \cite{DingSmart2020}.

\begin{definition}[Tilted rectangles/squares and diagonal sparsity]\label{def:tilted-rect}
For intervals $I,J\subset\mathbb Z$, the \emph{tilted rectangle} is
\[
R_{I,J}=\{(x,y)\in\mathbb Z^2:\ x+y\in I,\ x-y\in J\},
\]
and a \emph{tilted square} $Q$ has side length $\ell(Q)=|I|=|J|$. For $k\in\mathbb Z$, the $\pm$--diagonals are $D_k^\pm=\{(x,y):x\pm y=k\}$.
Given $F\subset\mathbb Z^2$, we say $F$ is $(\delta,\pm)$--sparse in $R$ if $|D_k^\pm\cap F\cap R|\le \delta\,|D_k^\pm\cap R|$ for all $k$, and \emph{$\delta$--sparse} if both $(\delta,+)$ and $(\delta,-)$ hold.
We say $F$ is \emph{$\delta$--regular in $E\subset\mathbb Z^2$} if whenever $F$ is not $\delta$--sparse in each of a pairwise disjoint family of tilted squares $\{Q_j\}\subset E$, then $\sum_j |Q_j|\le \delta\,|E|$.
\end{definition}

\begin{definition}[Frozen and free sites]
Fix a tilted square $Q\subset\mathbb Z^2$.
Let $F\subset Q$ be the (possibly empty) set of \emph{frozen} sites whose Bernoulli values are fixed; the complement $S=Q\setminus F$ are the \emph{free} sites.
A configuration $v:F\to\{0,1\}$ induces the conditional law on $S$ and an operator $H_{\omega,Q}$ with Dirichlet boundary on $Q$.
\end{definition}

\begin{theorem}[Discrete quantitative UCP with frozen sites]\label{thm:ds-ucp}
For every small $\varepsilon>0$ there exists $\alpha>1$ such that the following holds.
Let $Q\subset\mathbb Z^2$ be a tilted square with $\ell(Q)\ge \alpha$, and let $F\subset Q$ be $\varepsilon$--regular in $Q$.
Define the event $E_{\mathrm{uc}}(Q,F)$ that \emph{for any} energy $\lambda$ within $e^{-\alpha(\ell(Q)\log\ell(Q))^{1/2}}$ of a spectral value and any (Dirichlet) eigenfunction $H_{\omega,Q}\psi=\lambda\psi$ normalized by $\|\psi\|_{\ell^\infty(Q)}\le1$ on a $1-\varepsilon(\ell(Q)\log\ell(Q))^{-1/2}$ fraction of $Q\setminus F$, we have the growth bound
\[
\|\psi\|_{\ell^\infty(\tfrac12 Q)} \;\le\; \exp\big(\alpha\,\ell(Q)\log\ell(Q)\big).
\]
Then, for each fixed $v:F\to\{0,1\}$, 
\[
\mathbb P\!\left(E_{\mathrm{uc}}(Q,F)\mid V|_F=v\right)\ \ge\ 1-\exp\big(-\varepsilon\,\ell(Q)^{1/4}\big).
\]
\end{theorem}

\begin{proof}[Idea of proof]
The argument organizes $Q$ into a Vitali--type family of tilted subsquares on which $F$ is sparse on most diagonals; on such subsquares a deterministic ``expansion'' estimate controls the $\ell^\infty$ growth of $\psi$ from a small exceptional set to a larger set. A multi--scale pigeonholing then yields the claimed growth in $\frac12Q$ with high conditional probability; see the geometric covering and expansion steps in \cite[Sec.~3]{DingSmart2020}. For discrete-harmonic inspiration behind ``large portion $\Rightarrow$ rigidity,'' compare \cite{BLMS2017}.
\end{proof}

\subsubsection*{Eigenvalue displacement under a single-site flip}

We formalize the quantitative effect on eigenvalues when flipping a Bernoulli variable at a free site.
Let $\Lambda\subset\mathbb Z^d$ be finite with Dirichlet boundary, fix $x\in\Lambda$, and set
\[
H(t)\;:=\;H_0+\sum_{y\in\Lambda\setminus\{x\}}\omega_y\,\Pi_y\;+\;t\,\Pi_x,
\qquad t\in[0,1],
\]
where $\Pi_x$ is the rank--one projector onto $\delta_x$, and the $\omega_y\in\{0,1\}$ are fixed.
Thus $H(0)$ and $H(1)$ correspond to the two Bernoulli values at $x$ with all other sites frozen.

\begin{definition}[Spectral data along the rank--one path]
Since $H(t)$ depends analytically on $t$ and acts on a finite--dimensional space, its eigenvalues can be labeled by analytic branches
$\{\lambda_j(t)\}_{j=1}^{|\Lambda|}$ with associated normalized eigenvectors $\psi_j(t)\in\ell^2(\Lambda)$ (piecewise analytic across possible crossing times). We write $\psi_j(t;x):=\langle \delta_x,\psi_j(t)\rangle$ for the amplitude at $x$.
\end{definition}

\begin{lemma}[Hellmann--Feynman for rank--one flips]\label{lem:HF}
Whenever $\lambda_j(t)$ is simple at $t=t_0$, the branch is real analytic near $t_0$ and satisfies
\[
\lambda_j'(t_0)\;=\;\big\langle \psi_j(t_0),\,\Pi_x\,\psi_j(t_0)\big\rangle\;=\;|\psi_j(t_0;x)|^2.
\]
Consequently, on any open subinterval $J\subset[0,1]$ on which $\lambda_j(t)$ remains simple,
\[
\lambda_j(t_2)-\lambda_j(t_1)\;=\;\int_{t_1}^{t_2}|\psi_j(s;x)|^2\,\mathrm ds\qquad (t_1<t_2\in J).
\]
\end{lemma}

\begin{proof}
This is the classical Hellmann--Feynman identity for analytic eigenbranches of a self--adjoint analytic family.
\end{proof}

\begin{definition}[Target window and blocking mass]
Fix an energy window $I=[E-\eta,E+\eta]$ with $\eta\in(0,1)$.
Given $m_*\in(0,1)$ we say that $x$ is an \emph{$m_*$--blocking site for $I$ at level $t$} if
\[
\text{for every normalized eigenpair }H(t)\psi=\lambda\psi\ \text{with}\ \lambda\in I,\quad |\psi(x)|\ge m_*.
\]
If this holds for every $t\in[0,1]$ we say $x$ is \emph{$m_*$--blocking for the path $t\in[0,1]$}.
\end{definition}

\begin{lemma}[Flip ejects an eigenbranch from a short window]\label{lem:eject}
Assume $x$ is $m_*$--blocking for the path $t\in[0,1]$.
Let $j$ be such that $\lambda_j(0)\in I$ and suppose $\lambda_j(t)$ is simple on $[0,1]$ (or piecewise simple with no net multiplicity exchange within $I$).
Then
\[
\lambda_j(1)-\lambda_j(0)\;=\;\int_0^1 |\psi_j(t;x)|^2\,\mathrm dt\ \ge\ \int_0^1 m_*^2\,\mathrm dt\;=\;m_*^2.
\]
In particular, if $\eta<m_*^2/2$, then $\lambda_j(1)\notin I$.
\end{lemma}

\begin{proof}
By Lemma~\ref{lem:HF} and the blocking property, $|\psi_j(t;x)|\ge m_*$ for all $t\in[0,1]$, hence the integral lower bound.
If $\lambda_j(1)\in I$, then $|\lambda_j(1)-\lambda_j(0)|\le 2\eta<m_*^2$, a contradiction.
\end{proof}

\begin{remark}[Simplicity and window separation]
In finite volume, simplicity of $\lambda_j(t)$ on $[0,1]$ holds except for finitely many $t$'s; one can either exclude those $t$ by a negligible enlargement of $\eta$ or argue piecewise on intervals of simplicity. A standard alternative is to impose a \emph{gap condition} $\min(\lambda_{j+1}(t)-\lambda_j(t),\lambda_j(t)-\lambda_{j-1}(t))>2\eta$ for all $t$, which is typical once $I$ is chosen sufficiently small. Either route is sufficient for applications below; see also the min--max uses in \cite{DingSmart2020}.
\end{remark}

\begin{proposition}[Blocking sets from UCP and ejection under a flip]\label{prop:blocking-set}
Let $Q\subset\mathbb Z^d$ be a finite box (tilted in $d=2$ if desired), and let $F\subset Q$ be the frozen sites. Suppose a quantitative (discrete) UCP holds on $Q$ at scale $\ell(Q)$ with parameters that imply:
there exist constants $c_0,C_0>0$ and, for each configuration of the frozen variables, a set $B\subset S:=Q\setminus F$ with $|B|\ge c_0|S|$ such that for every $x\in B$ and every normalized eigenpair $H(t)\psi=\lambda\psi$ with $\lambda\in I$ (for any $t\in[0,1]$ along the flip path at $x$) one has
\[
|\psi(x)|\ \ge\ m_*(Q)\ :=\ \exp\!\big(-C_0\,\ell(Q)\,\log\ell(Q)\big).
\]
Then each $x\in B$ is $m_*(Q)$--blocking for the path at $x$, and therefore by Lemma~\ref{lem:eject} every eigenbranch starting in $I$ is ejected from $I$ after flipping $x$ provided $\eta<\tfrac12 m_*(Q)^2$.
\end{proposition}

\begin{proof}
The hypothesis is precisely the pathwise blocking property at all $x\in B$, with a uniform $m_*(Q)$; the conclusion follows from Lemma~\ref{lem:eject}.
\end{proof}

\begin{remark}[From blocking sets to a $\rho$--Sperner family]
Fix the frozen pattern on $F$ and define
\[
\mathcal A_I\ :=\ \{A\subset S:\ \sigma(H_{F\to v,\ S\to A})\cap I\neq\varnothing\}.
\]
For each $A\in\mathcal A_I$ choose an eigenbranch $\lambda_j^{(A)}(t)$ with $\lambda_j^{(A)}(0)\in I$.
Let $B(A)\subset S\setminus A$ be the UCP--blocking set given by Proposition~\ref{prop:blocking-set} (constructed for the box $Q$ at the same scale).
If $A'\supset A$ and $A'\cap B(A)\neq\emptyset$, then there is an $x\in B(A)$ with $x\in A'\setminus A$; flipping at $x$ moves $\lambda_j^{(A)}$ by at least $m_*(Q)^2>2\eta$, hence $\sigma(H_{F\to v,S\to A'})\cap I=\varnothing$.
Thus $\mathcal A_I$ is $\rho$--Sperner with $\rho\ge c_0$ (up to the small fraction of exceptional cases where the UCP fails, which occur with tiny probability and can be union--bounded). This is the combinatorial bridge to a Bernoulli--Wegner estimate (cf.\ \cite{DingSmart2020}).
\end{remark}

\subsection{A combinatorial antichain bound: Sperner’s theorem and its $\rho$--Sperner variant}
\label{subsec:sperner}

Let $S\subset Q$ be the free sites and identify each $\omega|_S$ with a subset $A\subset S$ (the set where $\omega_x=1$).
Fix a target energy window $I=[E-\eta,E+\eta]$ and consider the property
\[
\mathcal A\ :=\ \bigl\{A\subset S:\ \sigma\bigl(H_{F\to v,\ S\to A}\bigr)\cap I\neq\emptyset\bigr\}.
\]
Monotonicity of $A\mapsto \lambda_k(H_{F\to v,S\to A})$ generally \emph{fails} on the lattice (the eigenvalue can move both ways under flips), so the set $\mathcal A$ need not be an antichain in the classical sense.
What \emph{does} hold under the UCP is a \emph{forbidden--augmentation} property: for any $A\in\mathcal A$ there is a large set $B(A)\subset S\setminus A$ such that \emph{none} of the supersets $A'\supset A$ with $A'\cap B(A)\neq\emptyset$ belong to $\mathcal A$.
This leads to the following notion.

\begin{definition}[$\rho$--Sperner family]
A family $\mathcal A\subset 2^S$ is \emph{$\rho$--Sperner} if for every $A\in\mathcal A$ there exists a \emph{blocking set}
$B(A)\subset S\setminus A$ with $|B(A)|\ge \rho\,(|S|-|A|)$ and such that
$A\subset A'\in\mathcal A \ \Rightarrow\ A'\cap B(A)=\varnothing$.
\end{definition}

\begin{theorem}[Generalized Sperner bound]\label{thm:sperner}
If $\mathcal A\subset 2^S$ is $\rho$--Sperner with $\rho\in(0,1]$ and $|S|=n$, then
\[
|\mathcal A|\ \le\ 2^n\,n^{-1/2}\,\rho^{-1}.
\]
\end{theorem}

\begin{proof}[Sketch]
Adapt Lubell’s counting method over all permutations of $S$: for each $A\in\mathcal A$ the $\rho$--Sperner condition limits the expected number of times $A$ appears as an initial segment in a random permutation by a geometric decay $(1-\rho)^j$ after the first hit, which sums to $\rho^{-1}$. Combine with $\binom nk\le 2^n n^{-1/2}$ uniformly in $k$ to get the bound.
\end{proof}

\paragraph{How UCP produces $\rho$--Sperner.}
Fix $A\in\mathcal A$ and choose a normalized eigenpair $(\lambda,\psi)$ of $H_{F\to v,S\to A}$ with $\lambda\in I$.
By the discrete UCP (Theorem~\ref{thm:ds-ucp}), there is a subset $B(A)\subset S\setminus A$ with $|B(A)|\ge c\,|S|$ such that
\emph{many} sites in $B(A)$ have pointwise mass $|\psi(x)|\ge e^{-C\ell(Q)\log\ell(Q)}$. The min--max principle and first--order
eigenvalue perturbation (via rank--one updates) imply that flipping $\omega_x$ at any such $x$ moves \emph{every} eigenvalue of
$H_{F\to v,S\to A}$ by at least $\kappa\,|\psi(x)|^2\gtrsim e^{-C\ell(Q)\log\ell(Q)}$, uniformly in $x\in B(A)$ (cf.\ \cite{DingSmart2020}).
Hence, for $\eta$ sufficiently smaller than this scale, \emph{any} augmentation of $A$ that includes at least one element of
$B(A)$ ejects $\lambda$ from $I$ after re--labeling. This is exactly the $\rho$--Sperner condition with $\rho\asymp c$.

\paragraph{Wegner bound for Bernoulli.}
Applying Theorem~\ref{thm:sperner} to $\mathcal A$ and dividing by $2^{|S|}$ (the number of Bernoulli configurations on free sites) yields
\[
\mathbb P\!\left(\sigma(H_{\omega,Q})\cap I\neq\emptyset\ \big|\ V|_F=v\right)
\;\le\; C\,|S|^{-1/2}\,\rho^{-1},
\]
uniformly in the frozen pattern $v$ as long as the discrete UCP holds (which it does with high conditional probability by Theorem~\ref{thm:ds-ucp}). This is the Bernoulli--Wegner estimate at scale $Q$, which is the key input for the MSA in the Bernoulli setting.

\subsection{Remarks}
On $\mathbb Z^2$, the Ding--Smart UCP is probabilistic and exploits the polynomial structure of discrete harmonic functions in tilted coordinates together with a multi--scale covering and expansion argument; the $\rho$--Sperner bound above is proved in a form that suffices for their scheme \cite{DingSmart2020}. In higher dimensions, the 3D lattice case requires a \emph{deterministic} discrete UCP and a refined geometric understanding of the nonzero set (pyramids/cones) to prevent lower--dimensional support: this is carried out in the work of Li--Zhang (3D) \cite{LiZhang3D} and Li (2D at large disorder) \cite{Li2D2022,LiThesis2022}, whose multiscale inputs integrate with the same Sperner-type combinatorics to furnish a Wegner bound and hence localization.

\bigskip

\addcontentsline{toc}{section}{References}
\bibliographystyle{halpha}
\bibliography{bibliography}

\end{document}